\documentclass[ijoc,sglanonrev]{informs5}

\RequirePackage{tgtermes}
\RequirePackage{newtxtext}
\RequirePackage{newtxmath}
\RequirePackage{bm}
\RequirePackage{endnotes}

\OneAndAHalfSpacedXII

\usepackage{natbib}
\bibpunct[, ]{(}{)}{,}{a}{}{,}%
\def\bibfont{\small}%
\def\BIBand{and}%

\usepackage{mathtools}
\usepackage{bbm}
\usepackage{multirow}
\usepackage[para]{threeparttable}
\usepackage[caption=false]{subfig}
\usepackage{enumitem}
\usepackage{algorithm}
\usepackage{algpseudocode}
\usepackage{comment}
\usepackage{hyperref}

\EquationsNumberedThrough    

\TheoremsNumberedThrough     
\ECRepeatTheorems  %

\MANUSCRIPTNO{IJOC-0001-2026.00}

\DeclareMathOperator{\conv}{conv}
\DeclareMathOperator{\intset}{int}

\begin{document}


\RUNAUTHOR{Pathy and Rahimian}

\RUNTITLE{Intersection Cutting Plane Algorithm for Chance-Constrained Programs}

\TITLE{An Intersection Cutting Plane Algorithm for Chance-Constrained Programs with Finite Support}
\ARTICLEAUTHORS{%
\AUTHOR{Soumya Ranjan Pathy}
\AFF{Department of Industrial Engineering, Clemson University, Clemson SC 29634, USA  \EMAIL{spathy@clemson.edu}}

\AUTHOR{Hamed Rahimian}
\AFF{Department of Industrial Engineering, Clemson University, Clemson SC 29634, USA  \EMAIL{hrahimi@clemson.edu}}

} 

\ABSTRACT{%
Standard Big-M mixed-integer programming reformulations of chance-constrained programs (CCPs) with finite support often yield weak linear relaxations
and scale poorly as the number of scenarios increases. 
To address this limitation, we derive two families of intersection cuts from problem-specific S-free sets and develop
decomposition-based branch-and-cut frameworks. The first one exploits the submodular structure of the scenario requirement function, 
while the second one uses probability covers. 
We propose hybrid separation strategies that combine classical mixing inequalities and quantile cuts with the proposed intersection cuts. Computational experiments on integrated production-distribution planning instances show that the proposed hybrid methods reduce computational time compared to pure methods.
}%

\FUNDING{The authors gratefully acknowledge the support of the U.S. Air Force Office of Scientific Research through grant FA9550-24-1-0241.}



\maketitle

\section{Introduction}
Chance-constrained programming (CCP)  provides a framework for optimization under uncertainty by requiring stochastic constraints to hold with a prescribed high probability \citep{charnes1958cost,charnes1959chance,deng2021,dinh2018}. 
Over the past several decades, CCP has been applied in various domains, ranging from operations to engineering, see, e.g., \citet{bienstock2014,song2014}. 
The breadth of these applications underscores the practical importance of efficient CCP solution methods. 

When the underlying distribution of random parameters is finitely supported, the standard approach reformulates CCP as a deterministic mixed-integer program (MIP). It introduces binary scenario indicators together with a knapsack constraint. However, the resulting Big-M formulation has a notoriously weak linear programming (LP) relaxation. This has motivated extensive research on cutting planes to strengthen relaxations or to incorporate them into branch-and-cut algorithms \citep{luedtke2014branch,liu2016decomposition,pathy2024value}, thereby improving both computational efficiency and solution quality.

A foundational direction in this literature is the study of \textit{mixing} or \textit{star} inequalities, introduced by \cite{gunluk2001mixing}, to describe the convex hull of a mixing set in the form $\{(z,x) \in \mathbb{R}_+ \times \{0,1\}^m \mid x_j = 0 \implies z \ge h^j\}$. These inequalities became the foundation for later CCP developments.
The seminal application of mixing inequalities to CCPs is due to \cite{luedtke2010}. They observed that the single-row relaxation of a chance constraint with random right-hand sides contains a mixing set with a 0--1 knapsack constraint as a substructure. They derived \textit{strengthened star inequalities} that generalize classical mixing inequalities by exploiting the knapsack structure of the chance requirement. 
This single-row lens, where row-wise valid inequalities remain valid for the full system, has served as the foundation for deriving valid inequalities for the CCP over the past several decades 
\citep{kucukyavuz2012, abdi2015, zhao2017}. Subsequent work has also increasingly emphasized unifying perspectives on this family of inequalities. In particular, \cite{kilinc2021joint} showed that mixing inequalities are equivalent to polymatroid inequalities associated with specific submodular functions. Moreover, \cite{davarnia2025} proposed a convexification framework that reformulates the mixing set with a knapsack constraint as a bilinear set over a lifted simplex. 
Another notable line of research focuses on \textit{quantile} cuts, introduced by \citet{xie2016}, which can be viewed as projections of mixing inequalities from $(x,z)$-space to $x$-space.

Despite this progress, the dominant cut families for CCPs under finite support are still generated from comparatively fixed views of the mixing set. 
These views are powerful, but they do not directly use the local cone defined by the current fractional LP basis. 
Consequently, after the strongest algebraic or projection cuts have been added, the relaxation may tail off, with additional cuts providing limited changes in orientation or bound improvement. This effect becomes more pronounced as the number of scenarios grows. This motivates a complementary geometric mechanism that chooses cut orientations from the current fractional basis while preserving the probabilistic logic of the CCP.

In this paper, we pursue such a mechanism through \textit{intersection} cuts \citep{balas1971}. An intersection cut starts from a fractional basic solution, constructs the simplex cone generated by the non-basic tableau rays, and intersects these rays with the boundary of a convex \textit{S-free} set whose interior contains no feasible integer points. For CCPs, this construction is especially natural because the chance constraint itself induces S-free sets: the scenario-enforcement requirements yield submodular envelope inequalities, while the knapsack constraint yields probability-cover half-spaces. Intersection cuts therefore provide more than another class of valid inequalities; they translate the combinatorial violation logic of the chance constraint into basis-dependent cuts that complement mixing and quantile cuts. 

The strength of an intersection cut depends strongly on the chosen S-free set \citep{bienstock2020outer,modaresi2016}. 
Related to our work, \cite{fischetti2018} embedded intersection cuts in branch-and-cut methods for mixed-integer bilinear programs. More recently, \cite{xu2024} derived intersection cuts for submodular and submodular-supermodular functions by extending Lov\'{a}sz-type constructions to continuous piecewise-linear convex extensions whose epigraphs are valid S-free sets.

In this paper, we develop an intersection cutting-plane algorithm for CCPs with finite support in both the static (non-recourse) and two-stage (recourse) settings. We propose two ways to construct S-free sets: one based on the submodularity of (row-wise) scenario requirements (see Section \ref{sec:ic_sa}) and the other based on modular probability covers (see Section \ref{sec:ic_ma}). 
We then conduct an extensive computational comparison on integrated production-distribution planning instances. Results show that a hybrid switching strategy---starting with classical mixing inequalities and pivoting to intersection cuts when the optimality gap stalls---achieves the fastest convergence on large-scale instances.

The remainder of this paper is organized as follows. Section~\ref{sec:prelim_ic} introduces preliminaries on intersection cuts, including S-free sets and the generic intersection-cut formula. Section~\ref{sec:problem} formalizes CCP and its MIP reformulations for non-recourse and recourse settings. Section~\ref{sec:nonrecourse} presents the submodular and modular intersection cuts and their derivation for the non-recourse setting. Section~\ref{sec:recourse} extends these frameworks to the recourse setting. Section~\ref{sec:implment} describes algorithmic implementation and separation routines. Section~\ref{sec:numerical-c4} describes hybrid switching strategies and reports comparative computational results. Section~\ref{sec:conclusion-c4} concludes and outlines future research directions. All the proofs are relegated to the Electronic Companion (EC). 

\noindent \textbf{Notation.} 
For a set $S \subseteq \mathbb{R}^n$, we refer to the convex hull, interior, and boundary of $S$ by $\conv(S)$ $\intset(S)$, $\partial S$, respectively. 
Given $n \in \mathbb{N}$, we use $[n]$ to denote the index set $\{1,\dotsc,n\}$. We use $\mathbf{0}$ and $\mathbf{1}$ to denote a vector of 0's and 1's, respectively. For a function $f: \mathbb{R}^n \to \mathbb{R}$, $\mathrm{epi}(f)$ denotes the epigraph of $f$, where $\mathrm{epi}(f)=\{(x,y) \in \mathbb{R}^{n+1}: y \ge f(x)\}$. 

\section{Preliminaries on Intersection Cuts}
\label{sec:prelim_ic}

Let $P$ denote the LP relaxation feasible region and define
$S := \mathbb{Z}^p \times \mathbb{R}^{n-p}$. The mixed-integer feasible set is
$P_I = P \cap S$. Suppose $\bar{x} \in P \setminus P_I$ is an optimal basic feasible solution of
the LP relaxation. Fix its basis and let $J$ denote the non-basic index set. In tableau coordinates,
the simplicial cone rooted at $\bar{x}$ is
\begin{equation} \label{eq:simplex_cone}
  x = \bar{x} + \sum_{j \in J} r^j s_j,
  \qquad s_j \ge 0,\; j \in J,
\end{equation}
where $r^j \in \mathbb{R}^n$ is the extreme ray associated with non-basic variable $s_j$ \citep{gomory1969}.
Intersection cuts derive their validity from the identification of an S-free set---a convex
subset of the continuous relaxation space whose interior contains no integer-feasible
solutions \citep{balas1971, conforti2011, andersen2007}.

\begin{definition}[S-Free Set \citep{balas1971}]
\label{def:sfree_prelim}
Let $S \subseteq \mathbb{R}^n$. A closed convex set $B \subseteq \mathbb{R}^n$ is S-free if
$\intset(B) \cap S = \emptyset$.
\end{definition}

To construct an intersection cut, choose an S-free set $B$ with
$\bar{x} \in \intset(B)$. For each ray $r^j$, define its boundary exit distance by
\begin{equation*} \label{eq:step_length_prelim}
  \eta_j^*
  = \sup\bigl\{\eta \ge 0 : \bar{x} + \eta\, r^j \in B\bigr\}.
\end{equation*}
If a ray never exits $B$, then $\eta_j^* = +\infty$ and we use the convention
$1/\eta_j^* = 0$.

\begin{theorem}[Intersection Cut \citep{balas1971}]
\label{thm:ic_prelim}
The inequality
\begin{equation} \label{eq:ic_general}
  \sum_{j \in J} \frac{1}{\eta_j^*}\, s_j \ge 1
\end{equation}
is valid for $P_I$ and is violated by $\bar{x}$.
\end{theorem}

\section{Problem Definition}
\label{sec:problem}
We consider a generic CCP in which a decision vector
must satisfy uncertain constraints with high probability. Let
$x \in \mathbb{R}^n$ denote the decision vector, constrained to a deterministic set $X \subseteq \mathbb{R}_{+}^n$. 
Assume that $\xi$ is a random vector with a finite support on $N$ scenarios, i.e., $\xi \in \{\xi^1, \ldots, \xi^N \}$. Let $\Omega := \{1, \dots, N\}$. We have $\mathbb{P}(\xi = \xi^\omega) = p_\omega > 0$ for each $\omega \in \Omega$ and $\sum_{\omega \in \Omega} p_\omega = 1$.
For each scenario $\omega \in \Omega$, define $P_\omega := P(\xi^\omega)$. 
The following CCP describes the central problem under study in this paper
\begin{equation} \label{eq:ccp}
\min_{x} \left\{ c^\top x \;:\; \mathbb{P}\{ x \in P(\xi) \} \ge 1 - \epsilon, \; x \in X \right\},
\end{equation}
where $\epsilon \in (0, 1)$ is the risk tolerance. 
Moreover, we impose that $p_\omega \leq \epsilon$ for each $\omega$, since if $p_\omega > \epsilon$, then $x \in P_\omega$ must hold for any feasible $x$. In such cases, we can incorporate these inequalities directly into the definition of $X$ and exclude the corresponding scenario $\omega$ from further consideration.

We assume that $P_\omega$, $\omega \in \Omega$, is polyhedral.
We consider two common structures for $P_\omega$:

\noindent {\bf Non-recourse} setting, where  $P_\omega$ is the set of decisions that are feasible for realization $\xi^\omega$, with no recourse action, defined as
    \begin{equation} \label{eq:Pk_nonrecourse}
    P_\omega = \left\{ x \in \mathbb{R}^n \;:\; A x \ge b^\omega \right\}. 
    \end{equation}
    Here, $A \in \mathbb{R}^{m \times n}$ is the constraint matrix and
    $b^\omega \in \mathbb{R}_{+}^m$ is the nonnegative right-hand side vector under scenario $\omega \in \Omega$.

\noindent {\bf Recourse} setting, where corrective actions are allowed after uncertainty is realized. The set $P_\omega$ contains first-stage decisions for which a feasible scenario recourse action exists, defined as 
\begin{equation} \label{eq:Pk_recourse}
    P_\omega = \left\{ x \in \mathbb{R}^n \;:\; \exists\, y^\omega \in \mathbb{R}^d_+ \text{ s.t. } T^\omega x + W^\omega y^\omega \ge d^\omega \right\}. 
    \end{equation}
Here, $T^\omega \in \mathbb{R}^{m \times n}$ and $W^\omega \in \mathbb{R}^{m \times d}$ are,
respectively, the technology and recourse matrices under scenario $\omega$, and
$d^\omega \in \mathbb{R}^m$ is the corresponding right-hand side vector.

We make the following structural assumptions. 

{\bf A1 (Nonemptiness)} For each scenario $\omega \in \Omega$, $P_\omega$ is a nonempty polyhedron.

{\bf A2 (Common Recession Cone)} The polyhedra $P_\omega$ share a common recession cone. That is, there exists a cone $C \subseteq \mathbb{R}^n$ such that $C = \{r \in \mathbb{R}^n \mid x + \lambda r \in P_\omega\; \forall\, x \in P_\omega, \lambda \ge 0\}$ for all $\omega \in \Omega$.

Assumption~{\bf (A2)} holds, for example, when randomness appears only in the right-hand sides, so
that the left-hand-side matrices (e.g., $T$, $W$) are deterministic and fixed across
scenarios. In this case, we have $C= \{r:\sigma^\top T r \ge 0, \; \sigma \in \Sigma\}$, where $\Sigma:=\{\sigma: \sigma^\top W \le 0, \; \sigma \ge 0\}$. 
However, the developments in subsequent sections require only the common
recession-cone property and do not otherwise restrict left-hand-side stochasticity.

We now present a widely-used formulation of the CCP, defined in \eqref{eq:ccp}, as an MIP.  To encode the probabilistic constraint $\mathbb{P}\{ x \in P(\xi) \} \ge 1 - \epsilon$ in \eqref{eq:ccp}, for each $\omega \in \Omega$, we introduce a binary variable $\beta_\omega \in \{0, 1\}$, where $\beta_\omega=0$ enforces the constraint $x \in P_\omega$. 
Then, the total probability mass of violated scenarios
cannot exceed $\epsilon$. This yields the following equivalent logical mixed-integer formulation:
\begin{equation} \label{eq:det_equiv}
\begin{aligned}
\min_{x, \beta} \quad & c^\top x \\
\text{s.t.} \quad & x \in X, \\
& \beta_\omega = 0 \implies x \in P_\omega, \quad \forall\, \omega \in \Omega, \\
& \sum_{\omega \in \Omega} p_\omega \beta_\omega \le \epsilon, \\
& \beta \in \{0, 1\}^N.
\end{aligned}
\end{equation}
The implication $\beta_\omega = 0 \implies x \in P_\omega$ is often
linearized with Big-M coefficients, but the resulting LP relaxations are typically weak.
Instead, we maintain the compact logical representation and apply a decomposition-based cutting-plane method that exploits the underlying mixing-set substructure and the knapsack constraint to generate strong bounds. Similar approaches have been explored in the literature via the families of strengthened mixing inequalities, see, e.g., \cite{luedtke2014branch, pathy2024value, pathy2025decomposition}.

\section{Intersection Cuts for the Non-Recourse CCPs}
\label{sec:nonrecourse}

To address the weak LP relaxation bounds inherent in the Big-M formulation of CCPs, in this section, we develop geometric cutting-plane approaches for the non-recourse setting, defined via \eqref{eq:Pk_nonrecourse}.  Our proposed cutting-plane procedures are constructed based on intersection cuts that exploit the algebraic and combinatorial structure of the chance constraint to generate strong valid inequalities from the fractional extreme rays of the LP relaxation.  We propose two approaches that differ in the construction of the underlying S-free set.  The first approach, referred to as the 
\emph{submodular intersection cut}, is developed in Section~\ref{sec:ic_sa}.  
The second approach, referred to as the \emph{modular intersection cut}, is presented in Section~\ref{sec:ic_ma}. 

\subsection{MIP Reformulation of CCP}

\label{sec:mip}
Consider the non-recourse setting, defined via \eqref{eq:Pk_nonrecourse}.  
We know that $\beta_\omega=0$ enforces that $Ax \ge b^\omega$. Because $\epsilon < 1$, we must have $A x \ge b^\omega$  for at least one $\omega \in \Omega$, and because $b^\omega \ge 0$ for
all $\omega \in \Omega$, this implies $Ax \ge 0$ in every feasible solution of \eqref{eq:ccp}. Consequently, 
defining \(w = Ax\), we rewrite \eqref{eq:det_equiv} as
\begin{equation}
\label{form:chance_MIP}
\min_{x, w, \beta}    \Biggl\{ c^\top x : 
\begin{array}{l}
x \in X, \; w - Ax = 0, \; w \ge 0, \\
\beta_\omega = 0 \Rightarrow w \ge b^\omega, \; \forall\, \omega \in \Omega, \; \sum_{\omega \in \Omega} p_\omega \beta_\omega \le \epsilon, \;
\beta \in \{0, 1\}^N
\end{array}
\Biggr\}.
\end{equation}
Following the standard row-wise analysis in the mixing-set literature, we focus on the
\((w,\beta)\)-space set
\begin{equation*}
G = \Biggl\{ (w,\beta) \in \mathbb{R}_+^m \times \{0,1\}^N\, \Bigg| \,
\begin{array}{l}
\beta_\omega = 0 \Rightarrow w \ge b^\omega, \; \forall\, \omega \in \Omega, \; 
\sum_{\omega \in \Omega} p_\omega \beta_\omega \le \epsilon
\end{array}
\Biggr\}.
\end{equation*}
Because the logical constraints are separable by row, we can write
\begin{equation*}
G = \bigcap_{i \in [m]} \Bigl\{(w,\beta): (w_i,\beta) \in F_i\Bigr\},
\end{equation*}
where, for each \(i \in [m]\),
\begin{equation*}
F_i = \Biggl\{ (w_i,\beta) \in \mathbb{R}_+ \times \{0,1\}^N\, \Bigg| \,
\begin{array}{l}
\beta_\omega = 0 \Rightarrow w_i \ge b_i^{\omega}, \; \forall\, \omega \in \Omega, \; 
\sum_{\omega \in \Omega} p_\omega \beta_\omega \le \epsilon
\end{array}
\Biggr\}.
\end{equation*}
Hence, strengthening the relaxation of each \(F_i\) directly strengthens the relaxation of
\(G\). In particular, any valid inequality for \(\conv(F_i)\) yields a valid inequality for 
\(\conv(G)\), see, e.g., \cite{luedtke2010}.

To simplify notation, we drop the row index and study the generic set
\begin{equation} \label{eq:genprob_chance}
F = \Biggl\{ (y,\beta) \in \mathbb{R}_+ \times \{0,1\}^N\, \Bigg| \,
\begin{array}{l}
\beta_\omega = 0 \Rightarrow y \ge h^\omega, \; \forall\, \omega \in \Omega,\; 
\sum_{\omega \in \Omega} p_\omega \beta_\omega \le \epsilon
\end{array}
\Biggr\},
\end{equation}
replacing $w_i$ with $y$, and substituting $b^{\omega}_{i}$ with $h^{\omega}$ for each $i \in [m]$. 
We refer to the mixed-integer set \(F\) as the \textit{mixing set with a knapsack constraint}, for which we are interested in developing valid inequalities based on intersection cuts. Without loss of generality, assume the right-hand-side values are ordered as
\(h^1 \ge h^2 \ge \cdots \ge h^N\).

\subsection{Submodular Intersection Cut}
\label{sec:ic_sa}

This section develops a geometric cutting-plane approach by exploiting the structure of the right-hand-side values \( h^\omega \). By mapping the logical enforcement of scenario constraints to a submodular set function defined over a Boolean hypercube and lifting it to a
multi-faceted S-free set via the extended polymatroid and extended envelope frameworks \citep{edmonds1970, lovasz1983, xu2024}, we construct a convex extension over a continuous domain that yields a valid S-free set. Intersection cuts are then generated directly from the extreme rays of the simplex tableau of the LP relaxation of \eqref{form:chance_MIP}, as explained in Section \ref{sec:prelim_ic}.

\subsubsection{Submodular Set Function}
\label{sec:submodular_function}

Consider the logical constraints in \eqref{eq:genprob_chance}. 
For each $\omega \in \Omega$, each constraint can be modeled using a standard
Big-M formulation as 
\begin{equation*}
    y + M_{\omega} \beta_\omega \ge h^\omega, 
\end{equation*}
where 
\( M_{\omega} \) is a sufficiently large constant.
To analyze the underlying combinatorial structure, we 
apply a linear transformation to the binary indicators by setting \( z_\omega = 1 - \beta_\omega \in \{0, 1\} \), where \( z_\omega = 1 \) enforces $y \ge h^\omega$. 
Given \( y \ge 0 \) and letting \( M_{\omega} = h^\omega \), the constraint 
can be written as \( y \ge h^\omega z_\omega \). Consequently,
\( y \) must be bounded below by the maximum requirement value among all enforced
scenarios. We define a set function
\( f : \{0, 1\}^{N} \to \mathbb{R} \) over the Boolean hypercube as
\begin{equation} \label{eq:submod_function}
    f(z) = \max_{\omega \in \Omega:\, z_\omega = 1} h^\omega,
\end{equation}
with the convention \( f(\mathbf{0}) = 0 \). 
The following result establishes that the set function $f$ is submodular, as also observed by \cite{kilinc2021joint}.

\begin{proposition}[Submodularity of $f$]
\label{prop:submodularity}
The set function $f$, defined in \eqref{eq:submod_function}, is submodular over the Boolean
hypercube $\{0,1\}^{N}$.
\end{proposition}

\subsubsection{Extended Polymatroid and Envelope}
\label{sec:extended_polymatroid}

While the set function \( f \) accurately captures the feasibility structure \eqref{eq:genprob_chance}, the intersection-cut construction requires a continuous, convex region against which LP relaxation rays can be evaluated. We therefore extend the submodular function \( f \) continuously to
\( \mathbb{R}^{N} \) via the extended polymatroid framework \citep{nemhauser1978,lovasz1983,edmonds1970}, based on an associated polyhedron in the dual space of the scenario
indicators.

\begin{definition}[Extended Polymatroid]
\label{def:epm}
The extended polymatroid (EPM) of the submodular set function $f$, defined in \eqref{eq:submod_function}, is formed as
\begin{equation} \label{eq:epm}
    \mathrm{EPM}_f
    = \bigl\{ \pi \in \mathbb{R}^{N} : \pi^\top z \le f(z),
      \;\forall\, z \in \{0,1\}^{N} \bigr\}.
\end{equation}
\end{definition}

By the fundamental theorem of polymatroids \citep{edmonds1970}, the extreme points of
\( \mathrm{EPM}_f \), denoted \( \mathrm{ext}(\mathrm{EPM}_f) \), correspond precisely to the
sorted marginal gains of the requirement values \( h^\omega \), obtainable in \( O(N\log N) \) time via a greedy sorting algorithm.

\begin{definition}[Extended Envelope]
\label{def:ee}
The extended envelope (EE) of the submodular set function $f$, defined in \eqref{eq:submod_function}, is denoted as $\overline{F}_f: \mathbb{R}^{N} \to \mathbb{R}$, and is defined as 
\begin{equation} \label{eq:ee}
    \overline{F}_f(z)
    = \max_{\pi \in \mathrm{ext}(\mathrm{EPM}_f)} \pi^\top z.
\end{equation}
\end{definition}

The following properties of the extended envelope are established in \cite[Proposition~2 and Lemma~1]{xu2024}.
\begin{proposition}[Properties of the Extended Envelope]
\label{prop:ee_properties}
The extended envelope \( \overline{F}_f \), defined in \eqref{eq:ee}, has the following properties:
\begin{enumerate}
    \item \( \overline{F}_f\) is a continuous, piecewise-linear convex function over \( \mathbb{R}^{N} \), serving as a generalization of the classical Lov\'{a}sz extension \citep{lovasz1983}.
    \item For any \( z \in \{0,1\}^{N} \), \( \overline{F}_f(z) = f(z) \).
    \item $\mathrm{epi}(\overline{F}_f)= \conv\big(\mathrm{epi}(f)\big)$.  
\end{enumerate}
\end{proposition}

\subsubsection{Derivation of the Submodular Intersection Cut}
\label{sec:ic_sa_derivation}

We now translate the convex-extension results of Section~\ref{sec:extended_polymatroid}
into the geometric ingredients needed for cut generation. We first propose an S-free set for \eqref{eq:genprob_chance}, and then
use this set to define ray-boundary intersections from a fractional LP apex.

\begin{lemma}[Half-space as an S-free set]
\label{lem:sfree}
Consider $f$ defined in \eqref{eq:submod_function}, and let $\mathrm{EPM}_f$ and $\overline{F}_f$ be as defined in \eqref{eq:epm} and \eqref{eq:ee}. 
For any $\pi \in \mathrm{EPM}_f$, define $\mathcal{H}_{\pi} := \{(y,z) \in \mathbb{R}\times[0,1]^N : y \le  \pi^\top z\}
$. Then, $\mathcal{H}_{\pi}$ is S-free with respect to $\mathrm{epi}(f)$, i.e., $\intset\big(\mathcal{H}_{\pi}\big)\cap\mathrm{epi}(f)=\emptyset$. 
\end{lemma}

Suppose the LP relaxation of $F$, defined in \eqref{eq:genprob_chance}, yields an optimal basic fractional solution \( (\bar{y}, \bar{z}) \)
such that \( \bar{y} < \overline{F}_f(\bar{z}) \). By \eqref{eq:ee}, there exists $\bar{\pi} \in \mathrm{EPM}_f$ such that $\overline{F}_f(\bar{z})=\bar{\pi}^\top \bar z$. Therefore, \( (\bar{y}, \bar{z}) \in \intset(\mathcal{H}_{\bar{\pi}})\). We evaluate the local geometry of the corner polyhedron
rooted at this fractional apex. Let \( J \) be the
index set of non-basic variables \( s_j \ge 0 \) in the optimal simplex tableau, and let
\( (r_y^j, r_z^j) \) denote the corresponding extreme directions.

\begin{lemma}[Ray Exit Distance]
\label{lemma:ray_exit}
Let \( (\bar{y}, \bar{z}) \)
such that \( \bar{y} < \overline{F}_f(\bar{z}) \), where $\overline{F}_f(\bar{z})=\bar{\pi}^\top \bar z$ for some $\bar{\pi} \in \mathrm{EPM}_f$. 
The ray exit distance is
\begin{equation}
\label{eq:eta_sa}
    \eta_j^*= \frac{\bar{\pi}^\top \bar{z} -\bar y }{r_y^j - \bar{\pi}^\top  r_z^j}, 
\end{equation}
if $r_y^j - \bar{\pi} ^\top  r_z^j > 0$; otherwise, $\eta^*_j=+\infty$. 
\end{lemma}

Theorem~\ref{thm:ic_sa} is obtained by the ray-intersection construction above using the
S-free set \(\mathcal{H}_{\pi}\) from Lemma~\ref{lem:sfree}, with coefficients defined by the
ray-exit reciprocals, defined in \eqref{eq:eta_sa}.

\begin{theorem}[Submodular Intersection Cut]
\label{thm:ic_sa}
Let \( (\bar{x}, \bar{\beta}) \) be a fractional LP solution with
\( \bar{y} < \overline{F}_f(\bar{z}) \), 
where $\overline{F}_f(\bar{z})=\bar{\pi}^\top \bar z$ for some $\bar{\pi} \in \mathrm{EPM}_f$. Define 
\(\bar y = A_i \bar x\) and
\(\bar z = \mathbf{1} - \bar\beta\) for some $i \in [m]$. Let \(J\) be the index set of non-basic variables,
\(\{(r_x^j, r_\beta^j)\}_{j\in J}\) the corresponding extreme directions. 
Then, the inequality
\begin{equation} \label{eq:ic_sa}
    \sum_{j \in J}
    \max\!\left(0,\;
    \frac{A_i r_x^j + \bar{\pi} ^\top r_\beta^j}
         {\bar{\pi}^\top(\mathbf{1} - \bar{\beta}) - A_i \bar{x}}
    \right) \varphi^j(x, \beta) \ge 1,
\end{equation}
is valid for \eqref{eq:det_equiv}, where \( \varphi^j(x, \beta) \), $j \in J$, is the affine expression associated with non-basic variables. Moreover, the cut \eqref{eq:ic_sa} is violated by \( (\bar{x}, \bar{\beta}) \).
\end{theorem}

In Theorem~\ref{thm:ic_sa}, \(\varphi^j(x, \beta)\) is the affine tableau expression of non-basic variable \(s_j\), given by 
\begin{equation*}
(x,\beta)=(\bar x, \bar \beta)+ \sum_{j \in J} (r_x^j, r_\beta^j) s_j, \qquad j \in J.
\end{equation*}

\begin{remark}
\label{rem:hybrid}
Lemma~\ref{lemma:ray_exit} gives the closed-form exit distance for a fixed supporting
half-space \(\mathcal H_{\bar\pi}\), where \(\bar\pi\in\mathrm{EPM}_f\) satisfies
\(\overline F_f(\bar z)=\bar\pi^\top \bar z\). Since \(\overline F_f\) is the maximum of
finitely many affine functions induced by extreme points of \(\mathrm{EPM}_f\), the active
supporting vector can change as a tableau ray moves away from \((\bar y,\bar z)\). If one
tracks the envelope along a ray, the relevant residual is
\[
\zeta^j(\eta)
= \bigl(\bar y+\eta r_y^j\bigr)
  - \overline F_f\bigl(\bar z+\eta r_z^j\bigr), \qquad \eta\ge 0,
\]
which is piecewise-linear and concave. 
In our numerical experiments in Section~\ref{sec:numerical-c4}, we employ the hybrid discrete Newton algorithm of \cite{xu2024} to compute \( \eta_j^* \) in a finite number of steps by iteratively identifying the active extreme point of
\( \mathrm{EPM}_f \) and solving a local linear exit equation, see Algorithm \ref{algo:ic_sa_sep}. 
\end{remark}

\subsection{Modular Intersection Cut}
\label{sec:ic_ma}

We now present an alternative approach, the \emph{modular intersection cut}, which
constructs the S-free set directly from the combinatorial structure of the chance constraint,
without using the continuous requirement values $h^\omega$. We isolate the discrete feasibility structure of the
chance constraint through the notion of a \emph{probability cover}, yielding an S-free set that
is a simple half-space. 
As a side product, this approach circumvents the computational overhead of dynamically evaluating multi-faceted continuous
extensions, defined in \eqref{eq:ee}, required by the submodular approach in Section \ref{sec:ic_sa} (recall Remark \ref{rem:hybrid}).

\subsubsection{Probability Cover and Modular Coverage Function}
\label{sec:cover_function}

The central combinatorial object underlying the modular intersection cut is a \emph{probability cover} $K \subseteq \Omega$, defined as any set of scenarios whose aggregate probability exceeds the violation budget $\epsilon$. 

\begin{definition}[Probability Cover \citep{nemhauser1988integer}]
\label{def:prob_cover}
A set \( K \subseteq \Omega \) is called a probability if \( \sum_{\omega \in K} p_\omega > \epsilon \).
The cover \(K\) is minimal if no proper subset of \( K \) is a probability cover.
\end{definition}

The feasibility logic of a probability cover is straightforward: since the total probability mass
of \( K \) exceeds the violation budget \( \epsilon \), it is a strict logical necessity that at
least one scenario in \( K \) must be enforced, i.e., \( \beta_\omega = 0 \) for some
\( \omega \in K \). Using the enforcement
indicators \( z_\omega = 1 - \beta_\omega \), for any $(y,\beta) \in F$, we must have 
\( \sum_{\omega \in K} z_\omega \ge 1 \). Given a probability cover $K$, we define a set function $f_K: \{0,1\}^N \to \mathbb{R}_{+}$ over the Boolean hypercube as 
\begin{equation} \label{eq:modular_function}
    f_K(z) = \sum_{\omega \in K} z_\omega,
\end{equation}
with the convention \( f_K(\mathbf{0}) = 0 \).

\begin{proposition}[Modularity of \( f_K \)]
\label{prop:modularity}
Let \( K \subseteq \Omega \) be a probability cover. The function \( f_K \) defined in \eqref{eq:modular_function} is \emph{modular}, i.e., both
submodular and supermodular, over the Boolean hypercube \( \{0,1\}^{N} \).
\end{proposition}

\subsubsection{Derivation of the Modular Intersection Cut}
\label{sec:ic_ma_derivation}

The modularity of \( f_K \) has a critical computational consequence: unlike the submodular
function $f$ in Section~\ref{sec:ic_sa}, the Lov\'{a}sz extension of a modular
function is linear. Specifically, the convex extension of \( f_K \) over
\( \mathbb{R}^{N} \) is simply \( \overline{F}_{f_K}(z) = \sum_{\omega \in K} z_\omega \),
the same linear form. This means the S-free set boundary is a single static hyperplane, 
eliminating the need to compute the extended envelope through an optimization problem as in \eqref{eq:ee}. 

\begin{lemma}[Cover S-Free Set]
\label{lem:cover_sfree}
Let \( K \subseteq \Omega \) be a probability cover, and define the closed half-space
\begin{equation*} \label{eq:cover_sfree}
    \mathcal{H}_K
    = \left\{ \beta \in \mathbb{R}^{N} :
      \sum_{\omega \in K} \beta_\omega \ge |K| - 1 \right\}.
\end{equation*}
Then, \( \mathcal{H}_K \) is an S-free set with respect to $F$, i.e.,
\(\intset\big(\mathcal{H}_K\big) \cap F = \emptyset\).
\end{lemma}

Suppose the LP relaxation of $F$, defined in \eqref{eq:genprob_chance}, yields a fractional solution \( (\bar{x}, \bar{\beta}) \) in the interior of $\mathcal{H}_{K}$. 
Let \( J \) be the index set of non-basic variables \( s_j \ge 0 \) in the optimal
simplex tableau, and let \( r^j = (r_x^j, r_\beta^j) \) denote the extreme direction associated with
\( s_j \).

\begin{lemma}[Closed-Form Exit Distance for Modular Cut]
\label{lem:eta_ma}
Let \( (\bar{x}, \bar{\beta}) \in \intset(\mathcal{H}_{K})\). 
The ray exit distance is
\begin{equation} \label{eq:eta_ma}
    \eta_j^* = \frac{ \sum_{\omega \in K} \bar{\beta}_\omega - (|K| - 1)}{-\sum_{\omega \in K} r_{\beta,\omega}^j},
\end{equation}
if \( \sum_{\omega \in K} r_{\beta,\omega}^j < 0 \); otherwise, \( \eta_j^* = +\infty \).
\end{lemma}

\begin{theorem}[Modular Intersection Cut]
\label{thm:ic_ma}
Let \( (\bar{x}, \bar{\beta}) \) be a fractional LP solution with
$\sum_{\omega \in K} \bar{\beta}_\omega > (|K| - 1)$. Let \(J\) be the index set of non-basic variables,
and let \(\{(r_x^j, r_\beta^j)\}_{j\in J}\) be the corresponding extreme directions.
Then, the inequality
\begin{equation} \sum_{j \in J}
    \max\!\left(0,\;
    \frac{\sum_{\omega \in K} r_{\beta,\omega}^j}
         {(|K| - 1) - \sum_{\omega \in K}\bar{\beta}_\omega}
    \right)\varphi^j(x,\beta) \ge 1,
    \label{eq:ic_ma_main}
\end{equation}
is valid for \eqref{eq:det_equiv}, where \( \varphi^j(x, \beta) \), $j \in J$, is the affine expression associated with non-basic variables. Moreover, the cut \eqref{eq:ic_ma_main} is violated by \( (\bar{x}, \bar{\beta}) \).
\end{theorem}

\section{Intersection Cuts for the Recourse CCPs}
\label{sec:recourse}
In this section, we extend the intersection-cut framework of Section~\ref{sec:nonrecourse}
to the recourse setting, defined via  \eqref{eq:Pk_recourse}. For each scenario
\(\omega \in \Omega\), recall that
\begin{equation*}
P_\omega = \left\{ x \in \mathbb{R}^n \;:\; \exists\, y^\omega \in \mathbb{R}_+^d
\text{ such that } T^\omega x + W^\omega y^\omega \ge d^\omega \right\}.
\end{equation*}
We invoke Assumption~\textbf{(A2)} throughout, i.e., the sets
\(\{P_\omega\}_{\omega\in\Omega}\) share a common recession cone \(C\).
Let
\(C^* = \{\alpha \in \mathbb{R}^n : \alpha^\top r \ge 0,\ \forall r \in C\}\)
denote its dual cone. 
To illustrate this, suppose that matrices $T^{\omega}$ and $W^{\omega}$ are deterministic, i.e, $T^\omega=T$ and $W^\omega=W$ for all $\omega \in \Omega$. We then have $C^*=\{\sigma^\top T: \sigma \in \Sigma\}$, where $\Sigma:=\{\sigma: \sigma^\top W \le 0, \; \sigma \ge 0\}$.

Fix a closed, polyhedral set \(\bar{X} \subseteq \mathbb{R}^n\) with \(\bar{X} \supseteq X\) and  
\(P_\omega \cap \bar{X} \neq \emptyset\)  such that the recession cone of $ P_\omega \cap \bar{X}$ is contained in  $C$  for all \(\omega \in \Omega\).
Consider 
\begin{equation*} \label{eq:recourse_chance-c41}
G' = \Biggl\{ (x,\beta) \in \bar{X} \times \{0,1\}^N\, \Bigg| \,
\begin{array}{l}
\beta_\omega = 0 \Rightarrow x \in P_\omega, \; \forall\, \omega \in \Omega,\; 
\sum_{\omega \in \Omega} p_\omega \beta_\omega \le \epsilon
\end{array}
\Biggr\},
\end{equation*}
The key idea in this section is to solve single-scenario optimization subproblems and map the recourse
structure to the same mixing set with knapsack constraint \eqref{eq:genprob_chance}, used in
Section~\ref{sec:nonrecourse}. This enables direct use of the intersection cuts developed in Sections \ref{sec:ic_sa} and \ref{sec:ic_ma}.

For a given direction \(\alpha \in C^*\), define a single-scenario optimization problem as 
\begin{equation}
\label{eq:h_omega}
u^\omega(\alpha)
:=
\min\{\alpha^\top x : x \in P_\omega \cap \bar{X}\}
=
\min\{\alpha^\top x :
T^\omega x + W^\omega y^\omega \ge d^\omega,\ x \in \bar{X},\ y^\omega \ge 0\},
\quad \omega \in \Omega.
\end{equation}

\begin{proposition}
\label{prop:wellposed}
For any $\alpha \in C^*$, problem~\eqref{eq:h_omega} is feasible and
bounded for every $\omega \in \Omega$, and the inequality $\alpha^\top x \geq
u^\omega(\alpha)$ is valid for $P_\omega \cap \bar{X}$.
\end{proposition}

By Proposition \ref{prop:wellposed}, we have 
\begin{equation*}
G' \subseteq  \bigcap_{\alpha \in C^*} \Bigl\{(x,\beta): (x,\beta) \in F'({\alpha})\Bigr\},
\end{equation*}
where, for each \(\alpha \in C^*\),
\begin{equation*}
F'({\alpha}) = \Biggl\{ (x,\beta) \in \bar{X} \times \{0,1\}^N\, \Bigg| \,
\begin{array}{l}
\beta_\omega = 0 \Rightarrow \alpha^\top x \ge u^\omega(\alpha), \; \forall\, \omega \in \Omega, \; 
\sum_{\omega \in \Omega} p_\omega \beta_\omega \le \epsilon
\end{array}
\Biggr\}.
\end{equation*}
Hence, strengthening the relaxation of each \(F'(\alpha)\) directly strengthens the relaxation of
\(G'\). In particular, any valid inequality for \(\conv\big(F'(\alpha)\big)\) yields a valid inequality for 
\(\conv(G')\).
Therefore, for a fixed $\alpha \in C^*$, we study the generic set
\begin{equation} \label{eq:recourse_chance-c4}
F' = \Biggl\{ (y,\beta) \in \mathbb{R} \times \{0,1\}^N\, \Bigg| \,
\begin{array}{l}
\beta_\omega = 0 \Rightarrow y \ge h^\omega, \; \forall\, \omega \in \Omega, \;
\sum_{\omega \in \Omega} p_\omega \beta_\omega \le \epsilon
\end{array}
\Biggr\},
\end{equation}
setting $y:=\alpha^\top x$ and $h^\omega := u^\omega(\alpha)$, $\omega \in \Omega$.  Observe that \eqref{eq:recourse_chance-c4} 
has exactly the same structure as the mixing set with a knapsack constraint in
\eqref{eq:genprob_chance}. Thus, one can adopt the intersection cuts developed in Sections \ref{sec:ic_sa} and \ref{sec:ic_ma}. We skip the details for brevity.

\section{Solution Methodology and Algorithm Implementation}
\label{sec:implment}

Cut generation approaches, developed in Sections~\ref{sec:nonrecourse} and~\ref{sec:recourse}, can be seamlessly implemented programmatically within the callback framework of modern commercial MIP solvers, e.g., Gurobi, CPLEX. Because the proposed intersection cuts require querying the underlying basis of the LP relaxation at some nodes of the branch-and-bound tree, cut-separation routines are embedded inside a proper ``constraint"  callback. 

We outline the separation procedures in this section for the non-recourse setting at a nodal point $(\bar x,\bar \beta)$. 
Section~\ref{alg:ic_sa} explains these procedures for the submodular intersection cut, whereas Section~\ref{alg:ic_ma} explains them for the modular intersection cut.
These algorithms can be simply modified for the recourse setting \eqref{eq:Pk_recourse} at an integer incumbent solution $(\bar x, \bar \beta)$ instead of a fractional solution, by setting $h^{\omega}$ to $u^{\omega}(\alpha)$, as defined in \eqref{eq:h_omega}, and $y$ to $\alpha^\top x$, for some $\alpha \in C^*$.

\subsection{Separation Routines for Submodular Intersection Cut}
\label{alg:ic_sa}

Algorithm \ref{algo:ic_sa} presents the main separation routine for the submodular intersection cut, developed in Section \ref{sec:ic_sa}, for the non-recourse CCPs, defined via \eqref{eq:Pk_nonrecourse}. 
As part of the separation, extended
envelope \(\overline{F}_f\) at $\bar z=\mathbf{1} -\bar \beta$ is evaluated via a greedy approach by sorting scenarios in a nonincreasing order of \( \bar z_\omega \) and 
computing marginal gains of \( f \) along that ordering, see Algorithm \ref{algo:ic_sa_greedy}. Moreover, ray exit distances are also calculated via the hybrid discrete Newton scheme \citep{xu2024}, as outlined in Algorithm \ref{algo:ic_sa_sep}. Algorithms \ref{algo:ic_sa_greedy} and \ref{algo:ic_sa_sep} are relegated to Section~\ref{sec:EC_ic_sa}.

\begin{algorithm}[!ht]
\caption{Separation of Submodular Intersection Cuts.}
\label{algo:ic_sa}
\begin{algorithmic}[1]

\Statex \textbf{Input:} Fractional LP solution $(\bar{x}, \bar{\beta})$, constraint row $A_i$, Simplex tableau, with $J$ as the set of non-basic variable indices.

\Statex \textbf{Output:} \textsc{CUTFOUND}.

\State Set $h^\omega \leftarrow b_i^\omega$ for all $\omega \in \Omega$. 

\State Set $\bar{y} \leftarrow A_i \bar{x}$ and $\bar{z}_\omega \leftarrow 1 - \bar{\beta}_\omega$
       for all $\omega \in \Omega$.

\State $(\bar \nu, \bar \pi) \leftarrow$ Lov\'{a}sz $\textrm{Greedy}(\bar z, h)$ 
        \Comment{$\bar{\nu}$ and $\bar \pi$ at point $\bar{z}$}               

\State {Let $r= \{(r_y^j, r_z^j): j \in J\}$, where $r_y^j \leftarrow A_i r_x^j$ and $r_z^j \leftarrow -r_\beta^j$ and $(r_x^j$, $r_\beta^j)$ is an extreme direction from the tableau, $j \in J$.}

\If{$\bar{y} \ge \bar{\nu}$}
    \State \Return \textsc{CUTFOUND}=FALSE; \textbf{break}
    
\Else
    \State {\textsc{CUTFOUND} $\leftarrow \textrm{SepCut}_{SA}(\bar x, \bar \beta, h, \bar y, \bar z, \bar \pi, \bar \nu, J, r)$.}
    
\EndIf

\end{algorithmic}
\end{algorithm}

\subsection{Separation Routines for Modular Intersection Cut}
\label{alg:ic_ma}

Algorithm \ref{algo:ic_ma} presents the main separation routine for the modular intersection cut, developed in Section \ref{sec:ic_ma}, for the non-recourse CCPs, defined via \eqref{eq:Pk_nonrecourse}. As part of the separation, a minimal cover is constructed greedily by sorting scenarios in nonincreasing order of
\( \bar{\beta}_\omega \) and adding them until the probability budget \( \epsilon \) is
exceeded, see Algorithm \ref{algo:ic_ma_cover}. This heuristic minimizes the violation depth \(\Delta_K := (|K| - 1) - \sum_{\omega \in K} \bar{\beta}_\omega \)), among all covers, producing a valid cut whenever \( \Delta_K < 0 \). Moreover, ray exit distances are calculated as outlined in Algorithm \ref{algo:ic_ma_sep}. Algorithms \ref{algo:ic_ma_cover} and \ref{algo:ic_ma_sep} are relegated to Section~\ref{sec:EC_ic_ma}.

\begin{algorithm}[!htb]
\caption{Separation of Modular Intersection Cuts.}
\label{algo:ic_ma}
\begin{algorithmic}[1]

\Statex \textbf{Input:} Fractional LP solution $(\bar{x}, \bar{\beta})$, constraint row $A_i$, and Simplex tableau.
\Statex \textbf{Output:} \textsc{CUTFOUND}.

\State Set $\bar{y} \leftarrow A_i \bar{x}$ and $\bar{z}_\omega \leftarrow 1 - \bar{\beta}_\omega$
       for all $\omega \in \Omega$. 
       
\State Let $J$ be the set of non-basic variable indices from the LP tableau.

\For{$j \in J$}
    \State Extract rays from tableau: $r_x^j$, $r_\beta^j$.
\EndFor
\State{$r= \{(r_x^j, r_\beta^j): j \in J\}$.}

\State $(K,\Delta_K) \leftarrow$ Minimal $\textrm{Cover}(\bar \beta)$. 

\If{$\Delta_K \ge 0$}
    \State \Return \textsc{CUTFOUND}=FALSE; \textbf{break}
    
\Else
    \State {\textsc{CUTFOUND} $\leftarrow \textrm{SepCut}_{MA}(\bar x, \bar \beta, K, \Delta_K, J, r)$.}
    
\EndIf

\end{algorithmic}
\end{algorithm}

\begin{remark}
\label{rem:distinction}
    Unlike the separation for the submodular intersection cuts, the separation for the modular intersection cuts does not need to calculate ray exit distances via the hybrid discrete Newton scheme, as outlined in Algorithm \ref{algo:ic_sa_sep}. In particular, as the Lov\'{a}sz extension of the modular function \( f_K \), defined in \eqref{eq:modular_function}, is
\emph{linear}, the corresponding S-free set is a half-space; see, Lemma~\ref{lem:cover_sfree}. Thus,  the ray exit distance \( \eta_j^* \) is computed in
closed form by solving one linear equation; recall Lemma~\ref{lem:eta_ma}.
This entirely eliminates the need for iterative subgradient evaluation, required for the submodular intersection cut, and significantly accelerates the separation oracle (see Section \ref{sec:numerical-c4}). 
\end{remark}

\section{Numerical Study}
\label{sec:numerical-c4}

In this section, we report computational experiments on randomly generated instances of an integrated production--distribution planning problem, adapted from  \cite{luedtke2014branch}. 
All algorithms were implemented in Python 3.11 and solved by GUROBI 11.0.0 using a constraint callback function.  All experiments were executed on a single thread of high-performance computing nodes of the Palmetto Cluster with 32 cores and 125 GB of memory, with a time limit of 7200 seconds per instance and algorithm.
All codes are available on \url{https://github.com/Soumya-Pathy/Chance_Intersection}.

\subsection{Test Problem Formulations}
\label{sec:test_problems}

We consider an integrated production and distribution planning problem in a supply chain
consisting of a set of manufacturers $i \in \mathcal{I}$ and a set of retailers
$j \in \mathcal{J}$. Each retailer aggregates demand from a set of customers, and the retailer demand is the sum of the individual customer demands it represents. We develop
two CCPs corresponding to the structural settings defined in Section~\ref{sec:problem}.

\subsubsection{Non-Recourse Setting}
\label{sec:nr_formulation}

In this static setting, shipment allocations $x_{ij}$ must be determined
``here-and-now'' before demand is observed. No recourse variables are available to
correct shortfalls after uncertainty is realized. The CCP is modeled as 
\begin{equation*} \label{eq:nr_model}
\min_{x \ge 0} \left\{ \sum_{i \in \mathcal{I}} \sum_{j \in \mathcal{J}} c_i x_{ij} \;:\; \mathbb{P}\!\Big\{
  \sum_{i \in \mathcal{I}} d_{ij}\, x_{ij} \ge r_{j\omega},\;
  \forall\, j \in \mathcal{J}
\Big\} \ge 1 - \epsilon\right\}.
\end{equation*}
where $d_{ij} \in [0,1]$ is the effective delivery rate from manufacturer $i$ to
retailer $j$ (i.e., $1 - d_{ij}$ is the damage rate during transportation), and
$r_{j\omega}$ is the realized demand at retailer $j$ under scenario $\omega$.
This structure matches the random right-hand side form $Ax \ge b^\omega$ of
Section~\ref{sec:nonrecourse}, with single-row relaxations applied independently
for each retailer $j \in \mathcal{J}$.

\subsubsection{Recourse Setting}
\label{sec:r_formulation}

In this setting, first-stage variables $x_i \ge 0$ represent quantities
manufactured by manufacturer $i \in \mathcal{I}$ prior to demand realization.
Second-stage variables $y_{ij} \ge 0$ represent shipment quantities from
manufacturer $i$ to retailer $j$ dispatched after the scenario $\omega$ is
observed. Let $d_{ij} \in [0,1]$ denote the effective delivery rate from manufacturer
$i$ to retailer $j$, so that $(1 - d_{ij})$ is the corresponding damage (loss) rate
during transportation. The recourse-feasible set for scenario $\omega$ is 
\begin{equation*} \label{eq:S_omega}
S_\omega = \left\{ x \ge 0 \;:\;
  \exists\, y \ge 0 \text{ s.t. }
  \sum_{j \in \mathcal{J}} y_{ij} \le x_i,\, \forall\, i \in \mathcal{I},\quad
  \sum_{i \in \mathcal{I}} d_{ij}\, y_{ij} \ge r_{j\omega},\; \forall\, j \in \mathcal{J}
\right\},
\end{equation*}
which matches the recourse setting, defined via \eqref{eq:Pk_recourse}, with deterministic recourse and technology matrices. 
The first constraint ensures that total shipments from manufacturer $i$ do not exceed
its production quantity, and the second ensures that effective deliveries meet demand at
each retailer. The CCP is modeled as 
\begin{equation*} \label{eq:r_model}
\min_{x \ge 0} \left\{ \sum_{i \in \mathcal{I}} c_i x_i \;:\; \mathbb{P}\{x \in S_\omega\} \ge 1 - \epsilon \right\},
\end{equation*}

\subsection{Algorithmic Strategies and Hybrid Methods}
\label{sec:hybrids}

We evaluated the total average computational time across five different approaches: 
\begin{itemize}
    \item MI: Mixing inequalities, separated via sequence-independent lifting
    procedures \citep{luedtke2014branch, kucukyavuz2012}, within a branch-and-cut framework.
    \item IC-SA: Submodular intersection cut  (Algorithm~\ref{algo:ic_sa}). 
    \item IC-MA: Modular intersection cut (Algorithm~\ref{algo:ic_ma}). 
    \item Q: Quantile cuts \citep{xie2016}, within a branch-and-cut framework.
    \item DEF: Direct solution of the Big-M MIP reformulation of \eqref{eq:det_equiv} by GUROBI. 
\end{itemize}

To address the well-known ``tailing-off'' effect in pure cutting-plane approaches, we also evaluated two \emph{hybrid} strategies:
\begin{itemize}
    \item MI-IC(S): Initiates with MI, shifts to IC-MA when the optimality gap stalls
    (remains constant) for a predefined time limit.
    \item Q-IC(S): Initiates with Q, shifts to IC-MA when the optimality gap stalls
    for a predefined time limit.
\end{itemize}
In our preliminary computational experiments, we observed that the gap-triggered hybrid
approaches, 
\begin{itemize}
    \item MI-IC(G): Initiates with MI, shifts to IC-MA when the optimality gap drops below a predefined threshold percentage, 
    \item IC-MI(G): Initiates with IC-MA, shifts to MI when the optimality gap drops below a predefined threshold percentage,
    \item IC-MI(S): Initiates with IC-MA, shifts to MI when the optimality gap stalls for a predefined time limit, 
\end{itemize}
were not competitive relative to other methods. 
Therefore, we do not present these variants in the computational results.

\subsection{Computational Performance Results}

For our experiments, we considered problem dimensions $(|\mathcal{I}|, |\mathcal{J}|) \in \{(20, 30), (40, 50), (50, 100)\}$, sample sizes $N \in \{1000, 2000, 3000, 4000, 5000, 6000\}$, and risk parameters $\epsilon \in \{0.05, 0.1\}$. Detailed data generation is presented in Section \ref{sec:EC_data_generation}. 
Tables~\ref{tab:non_recourse} and \ref{tab:recourse} present the average computational times over five independent samples under the non-recourse and recourse settings, respectively. 
\begin{table}[!tb]
\small
\centering
\caption{Average computational time under the non-recourse setting (in seconds).}
\label{tab:non_recourse}
 \begin{threeparttable}
\begin{tabular}{cccllcclcc}
\hline
$(|\mathcal{I}|, |\mathcal{J}|)$ & $N$ & $\epsilon$ & DEF & MI & IC-SA & IC-MA & Q & MI-IC(S) & Q-IC(S)\\
\hline
\multirow{6}{*}{(20, 30)}
& 1000 & \multirow{6}{*}{0.05}
& 42.0 & \textbf{5.0} & 18.6 & 5.9 & 5.5 & 5.5 & 5.5\\
& 2000 &
& 175.0 & \textbf{15.0} & 51.0 & 16.6 & 15.7 & 15.2 & 15.3\\
& 3000 &
& 520.0 & 35.0 & 111.1 & 36.7 & 35.7 & \textbf{34.1} & 34.6\\
& 4000 &
& 1350.0 & 75.0 & 231.0 & 78.7 & 76.8 & \textbf{73.6} & 74.6\\
& 5000 &
& 3200.0 & 180.0 & 442.6 & 187.3 & 183.2 & \textbf{177.1} & 177.6\\
& 6000 &
& 5800.0 & 395.0 & 874.0 & 400.5 & 389.8 & 387.3 & \textbf{383.9}\\
\hline
\multirow{6}{*}{(40, 50)}
& 1000 & \multirow{6}{*}{0.05}
& 128.0 & \textbf{10.0} & 32.2 & 11.6 & 11.2 & 10.8 & 10.9\\
& 2000 &
& 480.0 & \textbf{28.0} & 78.0 & 30.2 & 29.4 & 28.1 & 28.5\\
& 3000 &
& 1420.0 & 62.0 & 164.9 & 64.7 & 63.4 & \textbf{60.8} & 61.6\\
& 4000 &
& 3250.0 & 120.0 & 347.3 & 123.8 & 121.5 & \textbf{116.7} & 118.1\\
& 5000 &
& 5600.0 & 265.0 & 656.6 & 269.0 & 261.4 & \textbf{219.0} & 231.7\\
& 6000 &
& 6850.0 & 580.0 & 1274.0 & 572.0 & 555.0 & \textbf{460.3} & 488.7\\
\hline
\multirow{6}{*}{(50, 100)}
& 1000 & \multirow{6}{*}{0.05}
& 375.0 & \textbf{22.0} & 63.7 & 24.6 & 23.8 & 22.3 & 22.8\\
& 2000 &
& 1400.0 & 55.0 & 145.3 & 59.7 & 58.3 & \textbf{54.8} & 55.8\\
& 3000 &
& 3650.0 & 115.0 & 317.1 & 120.9 & 118.6 & \textbf{112.0} & 114.0\\
& 4000 &
& 5500.0 & 210.0 & 587.6 & 214.1 & 211.9 & \textbf{203.3} & 205.9\\
& 5000 &
& 6800.0 & 450.0 & 1082.4 & 449.1 & 442.0 & \textbf{364.2} & 387.5\\
& 6000 &
& * & 920.0 & 2250.0 & 913.5 & 904.0 & \textbf{722.2} & 776.7\\
\hline
\multirow{6}{*}{(20, 30)}
& 1000 & \multirow{6}{*}{0.1}
& 36.0 & \textbf{6.0} & 22.3 & 7.0 & 6.7 & 6.6 & 6.6\\
& 2000 &
& 148.0 & \textbf{18.0} & 61.2 & 19.9 & 19.4 & 18.2 & 18.4\\
& 3000 &
& 440.0 & 42.0 & 133.3 & 44.1 & 42.9 & \textbf{40.9} & 41.5\\
& 4000 &
& 1150.0 & 90.0 & 277.2 & 94.4 & 92.2 & \textbf{88.4} & 89.6\\
& 5000 &
& 2750.0 & 216.0 & 531.1 & 224.8 & 219.9 & 211.9 & \textbf{210.2}\\
& 6000 &
& 5100.0 & 474.0 & 1048.8 & 480.6 & 468.7 & \textbf{455.9} & 461.1\\
\hline
\multirow{6}{*}{(40, 50)}
& 1000 & \multirow{6}{*}{0.1}
& 110.0 & \textbf{12.0} & 38.6 & 13.9 & 13.7 & 13.0 & 13.2\\
& 2000 &
& 410.0 & \textbf{33.6} & 93.6 & 36.2 & 35.5 & 33.8 & 34.5\\
& 3000 &
& 1200.0 & 74.4 & 197.9 & 77.6 & 76.3 & \textbf{73.0} & 73.9\\
& 4000 &
& 2800.0 & 144.0 & 416.8 & 148.5 & 146.9 & \textbf{140.1} & 142.3\\
& 5000 &
& 4850.0 & 318.0 & 788.0 & 322.8 & 316.0 & \textbf{262.8} & 277.7\\
& 6000 &
& 6200.0 & 696.0 & 1528.8 & 686.4 & 673.3 & \textbf{552.4} & 584.4\\
\hline
\multirow{6}{*}{(50, 100)}
& 1000 & \multirow{6}{*}{0.1}
& 320.0 & \textbf{26.4} & 76.4 & 29.6 & 28.9 & 26.8 & 27.2\\
& 2000 &
& 1200.0 & 66.0 & 174.3 & 71.6 & 70.2 & \textbf{65.7} & 66.6\\
& 3000 &
& 3150.0 & 138.0 & 380.5 & 145.1 & 142.6 & \textbf{134.5} & 136.7\\
& 4000 &
& 4800.0 & 252.0 & 705.1 & 257.0 & 254.4 & \textbf{243.9} & 247.3\\
& 5000 &
& 6100.0 & 540.0 & 1298.9 & 538.9 & 531.4 & \textbf{437.0} & 462.2\\
& 6000 &
& * & 1104.0 & 2700.0 & 1096.2 & 1079.8 & \textbf{866.6} & 918.0\\
\hline
\end{tabular}
\begin{tablenotes}
    \item * Indicates that the problem was not solved within the time limit.
\end{tablenotes}
\end{threeparttable}
\end{table}

\begin{table}[!tb]
\small
\centering
\caption{Average computational time under the recourse setting (in seconds).}
\label{tab:recourse}
\begin{threeparttable}
\begin{tabular}{cccllcclcc}
\hline
$(|\mathcal{I}|, |\mathcal{J}|)$ & $N$ & $\epsilon$ & DEF & MI & IC-SA & IC-MA & Q & MI-IC(S) & Q-IC(S)\\
\hline
\multirow{6}{*}{(20, 30)}
& 1000 & \multirow{6}{*}{0.05}
& 55.0 & \textbf{6.5} & 24.2 & 7.6 & 8.2 & 7.1 & 7.2 \\
& 2000 &
& 228.0 & \textbf{19.5} & 66.4 & 21.6 & 23.0 & 19.7 & 20.1 \\
& 3000 &
& 676.0 & 45.5 & 144.5 & 47.8 & 50.0 & \textbf{44.3} & 45.1 \\
& 4000 &
& 1755.0 & 97.5 & 300.3 & 102.3 & 105.0 & \textbf{95.7} & 96.9 \\
& 5000 &
& 4160.0 & 234.0 & 575.3 & 243.5 & 248.0 & \textbf{228.0} & 229.4 \\
& 6000 &
& 6950.0 & 513.5 & 1136.2 & 520.6 & 528.0 & \textbf{498.2} & 501.0 \\
\hline
\multirow{6}{*}{(40, 50)}
& 1000 & \multirow{6}{*}{0.05}
& 166.0 & \textbf{13.0} & 41.9 & 15.0 & 15.8 & 14.0 & 14.3\\
& 2000 &
& 624.0 & \textbf{36.4} & 101.4 & 39.2 & 40.2 & 36.6 & 37.1\\
& 3000 &
& 1846.0 & 80.6 & 214.3 & 84.1 & 85.9 & \textbf{79.1} & 80.4\\
& 4000 &
& 4225.0 & 156.0 & 451.5 & 160.9 & 162.7 & \textbf{151.8} & 154.2\\
& 5000 &
& 6720.0 & 344.5 & 853.6 & 349.7 & 352.4 & \textbf{284.7} & 296.5\\
& 6000 &
& 7100.0 & 754.0 & 1656.2 & 743.6 & 752.5 & \textbf{598.5} & 625.1\\
\hline
\multirow{6}{*}{(50, 100)}
& 1000 & \multirow{6}{*}{0.05}
& 488.0 & \textbf{28.6} & 82.8 & 32.0 & 32.8 & 29.0 & 29.5\\
& 2000 &
& 1820.0 & 71.5 & 188.9 & 77.6 & 78.8 & \textbf{71.2} & 72.5\\
& 3000 &
& 4745.0 & 149.5 & 412.2 & 157.2 & 159.0 & \textbf{145.7} & 147.9\\
& 4000 &
& 6600.0 & 273.0 & 763.9 & 278.4 & 280.4 & \textbf{264.3} & 267.7\\
& 5000 &
& 7050.0 & 585.0 & 1407.1 & 583.8 & 589.4 & \textbf{473.4} & 489.7\\
& 6000 &
& * & 1196.0 & 2925.0 & 1187.5 & 1198.3 & \textbf{938.9} & 971.7\\
\hline
\multirow{6}{*}{(20, 30)}
& 1000 & \multirow{6}{*}{0.1}
& 47.0 & \textbf{7.8} & 29.0 & 9.1 & 9.8 & 8.5 & 8.7\\
& 2000 &
& 192.0 & \textbf{23.4} & 79.6 & 25.9 & 27.5 & 23.6 & 24.1\\
& 3000 &
& 572.0 & 54.6 & 173.3 & 57.3 & 59.0 & \textbf{53.1} & 54.3\\
& 4000 &
& 1495.0 & 117.0 & 360.4 & 122.7 & 125.5 & \textbf{114.8} & 116.7\\
& 5000 &
& 3575.0 & 280.8 & 690.4 & 292.2 & 297.0 & 277.9 & \textbf{273.3}\\
& 6000 &
& 6250.0 & 616.2 & 1363.4 & 624.8 & 636.0 & \textbf{596.9} & 602.0\\
\hline
\multirow{6}{*}{(40, 50)}
& 1000 & \multirow{6}{*}{0.1}
& 143.0 & \textbf{15.6} & 50.2 & 18.0 & 18.7 & 16.8 & 17.0\\
& 2000 &
& 533.0 & \textbf{43.6} & 121.7 & 47.1 & 48.2 & 43.9 & 44.6\\
& 3000 &
& 1560.0 & 96.7 & 257.2 & 100.8 & 102.2 & \textbf{94.9} & 96.5\\
& 4000 &
& 3640.0 & 187.2 & 541.8 & 193.1 & 195.0 & \textbf{182.1} & 185.5\\
& 5000 &
& 5880.0 & 413.4 & 1024.4 & 419.7 & 423.4 & \textbf{341.7} & 359.3\\
& 6000 &
& 7050.0 & 904.8 & 1987.4 & 892.3 & 901.4 & \textbf{718.1} & 751.7\\
\hline
\multirow{6}{*}{(50, 100)}
& 1000 & \multirow{6}{*}{0.1}
& 416.0 & \textbf{34.3} & 99.4 & 38.4 & 39.1 & 34.8 & 35.3\\
& 2000 &
& 1560.0 & 85.8 & 226.6 & 93.1 & 94.3 & \textbf{85.4} & 87.3\\
& 3000 &
& 4095.0 & 179.4 & 494.7 & 188.5 & 190.1 & \textbf{174.7} & 177.6\\
& 4000 &
& 5760.0 & 327.6 & 916.7 & 334.0 & 337.0 & \textbf{317.1} & 321.3\\
& 5000 &
& 6830.0 & 702.0 & 1688.5 & 700.6 & 706.1 & \textbf{568.1} & 585.9\\
& 6000 &
& * & 1435.2 & 3510.0 & 1425.1 & 1436.5 & \textbf{1126.6} & 1169.1\\
\hline
\end{tabular}
\begin{tablenotes}
    \item * Indicates that the problem was not solved within the time limit.
\end{tablenotes}
\end{threeparttable}
\end{table}

Observe from  Tables~\ref{tab:non_recourse} and~\ref{tab:recourse} that IC-SA is consistently slower than the other approaches. This is because IC-SA evaluates the extended envelope and identifies an active subgradient while computing the intersection distance along each tableau ray. In contrast, IC-MA uses the modular probability cover structure and computes the ray exit distance in closed form, thereby avoiding the iterative subgradient search required by IC-SA (recall Remark \ref{rem:distinction}).

For smaller and moderately sized instances, the standalone MI approach performs well and often gives the smallest solution time. This is expected because the mixing inequality separation is carried out in the original variable space and does not require extracting tableau rays. The standalone Q approach shows a similar trend and remains close to MI, since quantile cuts also exploit the ordered scenario structure without requiring tableau-based computations. However, Q is generally slightly slower than MI, suggesting that mixing inequalities provide stronger initial tightening. As the number of scenarios increases, the advantage of both MI and Q decreases. 
The hybrid method MI-IC(S) performs better in the larger instances. The numerical results indicate that MI is effective in the early stage of the solution process, where it quickly reduces the optimality gap. Once the gap improvement stalls, switching to IC-MA introduces stronger intersection cuts that further tighten the relaxation. This explains why MI-IC(S) gives the fastest or near-fastest solution times for most large instances. A similar behavior is observed for Q-IC(S), which remains competitive because it also benefits from switching to IC-MA after the initial cut family loses effectiveness.

We also observe that the recourse instances require more computational time than the corresponding non-recourse instances. This is due to the additional Farkas separation step needed to project the second-stage feasibility condition onto the first-stage variable space. Although this step increases the computational burden, the relative performance of the methods remains similar in the non-recourse and recourse settings. Overall, the results show that mixing inequalities are effective for smaller and moderate-sized instances, while the stall-triggered hybrid methods become more useful as the number of scenarios increases.

\subsection{Performance Profiles}

To compare the convergence behavior of the seven solution approaches beyond aggregate
computational times, we constructed performance profiles for both the non-recourse
and recourse settings.

For each parameter configuration $(|\mathcal{I}|,|\mathcal{J}|,N,\varepsilon)$, totaling 36 configurations, we generated five independent
random samples, yielding a total of $S=180$ instances.  
We solved each instance using all seven approaches. During every solve, the optimality gap is recorded at 2-second time intervals, producing a trajectory
$g_a^s(t)$ for approach $a$ on instance $s$ at elapsed time $t$ until either the instance is solved to optimality or time limit of 7200 seconds. To enable meaningful
cross-instance comparison, 
we calculated a normalized performance score
\[
P_a^s(t)
=
\frac{g_a^s(t) - \min_{a} g_a^s(t)}{\max_{a} g_a^s(t) - \min_{a} g_a^s(t)},
\]
so that $P_a^s(t) \in [0,1]$, with values near $0$ indicating near-best performance and
values near $1$ indicating comparatively poor performance at time $t$. 
We then defined the empirical cumulative distribution function (cdf)
\[
F_a(\tau,t)
=
\frac{1}{S} \sum_{s=1}^{S} \mathbf{1}\!\left[ P_a^s(t) \le \tau \right],
\qquad \tau \in [0,1],
\]
and created computational gap performance profiles for each approach. 

Figure~\ref{fig:cdf_profiles} shows the empirical cdfs $F_a(\tau, t)$ against elapsed computational
time $t$ at a fixed threshold $\tau=0.2$ for all approaches, under the non-recourse and recourse settings. The vertical axis, therefore, reports the
probability that the approach $a$ is within the near-best normalized gap region at time $t$,
while the horizontal axis extends to $1500$ seconds. Curves that rise steeply indicate
approaches reaching low optimality gaps early and consistently; flatter curves signal
slower convergence or greater variability. The relative ordering and separation of the
curves provide a distributional view of convergence that complements the aggregate times
in Tables~\ref{tab:non_recourse} and~\ref{tab:recourse}.

\begin{figure}[htbp]
\centering
\subfloat[Non-recourse setting.]{
    \includegraphics[width=0.48\textwidth]{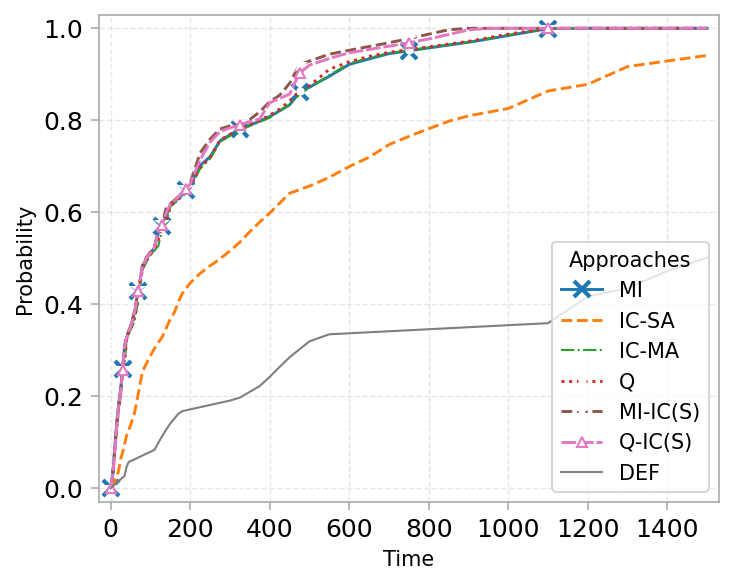}
    \label{fig:cdf_nonrecourse}
}
\hfill
\subfloat[Recourse setting.]{
    \includegraphics[width=0.48\textwidth]{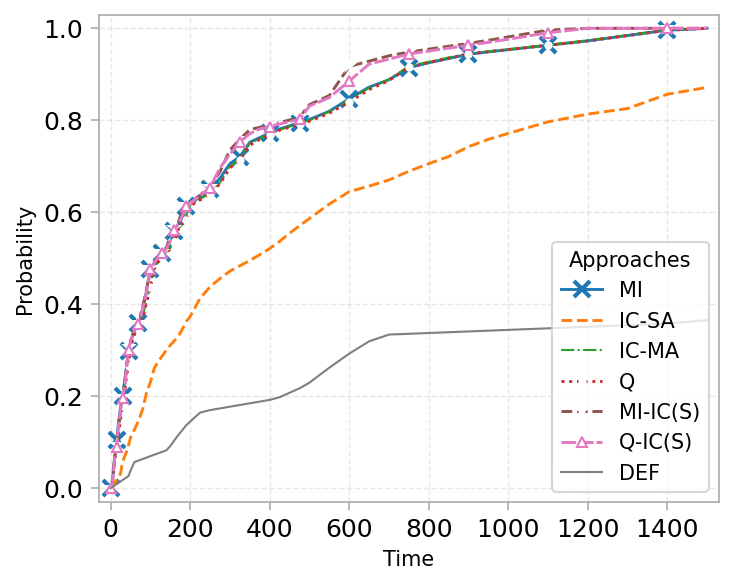}
    \label{fig:cdf_recourse}
}
\caption{Optimality gap performance profiles at
  threshold $\tau = 0.2$ over a time horizon of $1500$ seconds.}
\label{fig:cdf_profiles}
\end{figure}

As a complementary analysis, we also constructed a computational time performance profile that normalizes total solve times rather than optimality gap trajectories. For each instance $s$, let $t_a^s$ denote the total solve time of approach $a$. We define the normalized time score
\[
  P_a^s
  = \frac{t_a^s - \min_{a} t_{a}^s}{\max_{a} t_{a}^s - \min_{a} t_{a}^s},
\]
so that $P_a^s = 0$ for the fastest approach on instance $s$ and $P_a^s = 1$ for the slowest. The empirical cdf
\[
  F_a(\tau)
  = \frac{1}{S}\sum_{s=1}^{S}
    \mathbf{1}\!\left[P_a^s \le \tau\right],
\]
then reports, for each threshold $\tau \in [0,1]$, the fraction of instances on which $P_a^s$ for approach $a$ is at most $\tau$. By definition, the point $\big(0,F_a(0)\big)$ implies that $F_a(0)$ fraction of instances were solved quickest by approach $a$.

Figure~\ref{fig:profiles} displays the computational time performance profiles for the non-recourse and recourse settings. In both settings, MI and MI-IC(S) dominate the left portion of the profile, confirming that they are the fastest or near-fastest methods on the vast majority of instances. 
IC-SA is consistently the slowest approach, with its profile remaining at zero until $\tau = 0.6$, indicating that it is the maximum-time method on every instance tested. These profiles corroborate the aggregate timing results of Tables~\ref{tab:non_recourse} and~\ref{tab:recourse} and further highlight the practical advantage of the hybrid switching strategies at scale.

\begin{figure}[htbp]
\centering
\subfloat[Non-recourse setting.]{
    \includegraphics[width=0.48\textwidth]{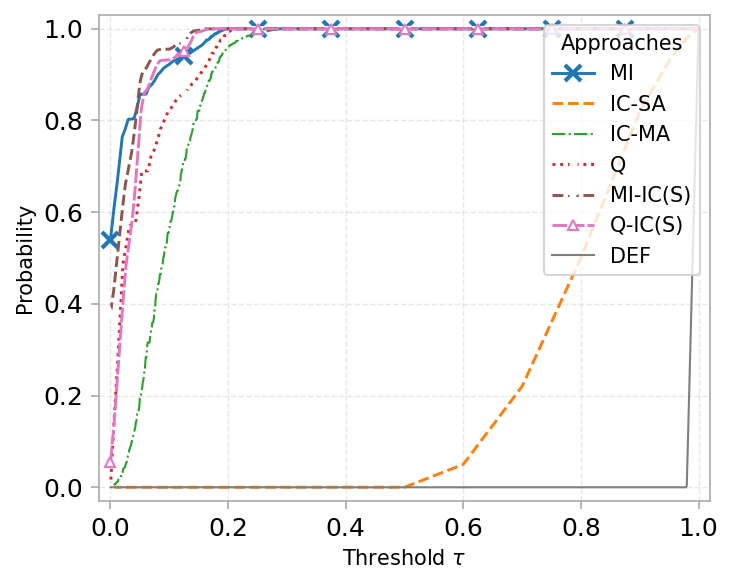}
    \label{fig:profile_nr}
}
\hfill
\subfloat[Recourse setting.]{
    \includegraphics[width=0.48\textwidth]{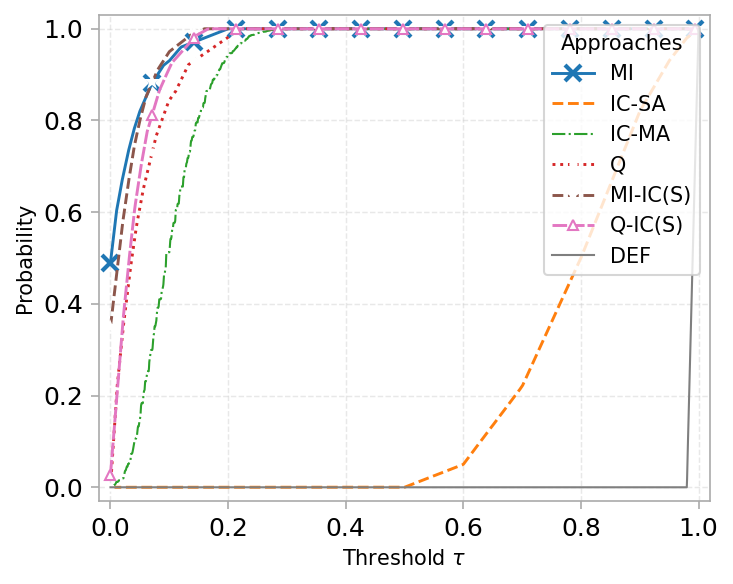}
    \label{fig:profile_rec}
}
\caption{Computational time performance profiles.}
\label{fig:profiles}
\end{figure}

\section{Conclusion}
\label{sec:conclusion-c4}

This paper developed a decomposition-based branch-and-cut framework for solving 
chance-constrained programs (CCPs) under finite support, addressing 
both static non-recourse and two-stage recourse settings. To overcome the well-documented 
weaknesses of the Big-M MIP reformulations, we derived two 
geometrically distinct families of intersection cuts. The first, the submodular intersection cut, IC-SA, establishes that the scenario requirement mapping, $f(z) = \max_{\omega: z_\omega=1}  h^\omega$, is submodular over the Boolean hypercube, constructs its continuous convex extension via the extended polymatroid and extended envelope, and derives cut coefficients through a hybrid discrete Newton algorithm whose exit distances are computed against the resulting S-free set. The second, the modular intersection cut, IC-MA, exploits the strict modularity of the function $f_K(z) = \sum_{\omega \in K} z_\omega$ based on the cover $K$, which reduces the S-free boundary to a single static hyperplane and admits a closed-form analytic step length, entirely eliminating the iterative subgradient search required by the former approach. For two-stage recourse formulations, Farkas' Lemma projects the second-stage feasibility conditions onto the first-stage space via a scenario-independent dual extreme point, reducing the recourse problem to an equivalent non-recourse structure that admits direct application of both cut families 
without inflating the variable space.

Computational experiments reveal a clear performance hierarchy. IC-SA is 
consistently the slowest method due to the overhead of dynamic subgradient evaluation at every LP node, confirming that theoretical tightness alone does not imply computational efficiency. Classical mixing inequalities (MI) dominate in small-to-medium scale instances.
However, on a large scale, pure MI suffers from the tailing effect, which causes the optimality gap to stagnate over an extended branch-and-bound 
exploration. In this regime, the time-stall-triggered hybrid MI-IC(S), which initializes 
with MI for rapid bound tightening and pivots to IC when the gap fails to 
improve within a threshold duration, typically achieves the fastest convergence. 
Recourse instances require greater computation time than non-recourse counterparts at every 
configuration, reflecting the additional per-node LP overhead of the Farkas ray separation step.

Future research can proceed in both theoretical and computational directions. On the theoretical side, an important open question is to characterize the closure obtained by intersecting all valid intersection cuts generated from CCP-induced S-free sets. In particular, it remains unclear whether this intersection closure admits a finite polyhedral description for linear CCPs, how it compares with mixing and quantile closures, and whether repeated closure operations converge to the convex hull of the chance-constrained feasible region. On the computational side, the results suggest the need for adaptive procedures that select among submodular, modular, and classical mixing cuts based on the local LP geometry and the observed progress of the branch-and-cut algorithm. Extending these ideas to distributionally robust chance constraints and to broader recourse structures is another promising direction.

\ACKNOWLEDGMENT{The authors gratefully acknowledge the support of the U.S. Air Force Office of Scientific Research through grant FA9550-24-1-0241. This research used in part resources on the Palmetto Cluster at Clemson University under National Science Foundation awards MRI 1228312, II NEW 1405767, MRI 1725573, and MRI 2018069. The views expressed in this article do not necessarily represent the views of NSF or the United States government.}


\ECSwitch 

\ECHead{Electronic Companion}



\section{Proofs}

\subsection{Proofs of Section \ref{sec:prelim_ic}}

\begin{proof}{Proof of Theorem \ref{thm:ic_prelim}.}
Let $x \in P_I$ and write $x = \bar{x} + \sum_{j \in J} r^j s_j$ with $s_j \ge 0$.
Suppose, for contradiction, that $\sum_{j \in J} s_j/\eta_j^* < 1$ for some $x \in P_I$.
Suppose for all $j \in J$, we have $\eta_j^*<\infty$. Define $t^j := \bar{x} + \eta_j^* r^j \in \partial B$ and let  
$\lambda_j := \frac{s_j}{\eta_j^*}$.
Let $\lambda_0 := 1-\sum_{j \in J}\lambda_j > 0$. Then, $x
  = \bar{x} + \sum_{j \in J} s_j r^j
  = \lambda_0\bar{x} + \sum_{j \in J} \lambda_j t^j$.
That is, $x$ is a convex combination of $\bar{x}\in \intset(B)$ and boundary points $t^j\in B$.
The convexity of $B$ implies $x\in B$. Moreover, because $\lambda_0>0$ and $\bar{x} \in \intset(B)$, we have $x\in \intset(B)$, contradicting $x\in P_I\subseteq S$ and
$\intset(B)\cap S=\emptyset$. Therefore $\sum_{j \in J} s_j/\eta_j^* \ge 1$ for every
$x \in P_I$. For rays with $\eta_j^*=\infty$, we use the convention $1/\eta_j^*=0$ as stated above. 
At the fractional apex $\bar{x}$, we have $s_j=0$ for all $j$, so the left-hand side of
\eqref{eq:ic_general} is $0<1$; hence the cut separates $\bar{x}$.
\hfill \Halmos \end{proof}

\subsection{Proofs of Section \ref{sec:nonrecourse}}

\begin{proof}{Proof of Proposition \ref{prop:submodularity}.}
For any $z \in \{0,1\}^N$, let $U$ be the set of indices where $z_{\omega}=1$, i.e., $U=\{\omega \in \Omega: z_{\omega}=1\}$. That is, there is a one-to-one mapping between $z$ and a set $U$. Hence, $f(z)\equiv f(U)$. A set function is submodular if and only if it satisfies the diminishing returns property: for all \( U \subseteq V \subseteq \Omega \) and any
\( \omega' \notin V \), we require \( f(U \cup \{\omega'\}) - f(U) \ge f(V \cup \{\omega'\}) - f(V) \).

Suppose \( \max_{j \in V} h^j \ge h^{\omega'} \). Then, 
\( f(V \cup \{\omega'\}) - f(V) = 0 \). Since \( f(\cdot) \) is monotone nondecreasing, we have
\( f(U \cup \{\omega'\}) - f(U) \ge 0 \). Hence,
\( f(U \cup \{\omega'\}) - f(U) \ge f(V \cup \{\omega'\}) - f(V) \).

On the other hand, suppose \( \max_{j \in V} h^j < h^{\omega'} \). Then, 
\( f(V \cup \{\omega'\}) - f(V) = h^{\omega'} - f(V) \). Because \( U \subseteq V \), we have
\( f(U) \le f(V) \), from which it follows that
\( h^{\omega'} - f(U) \ge h^{\omega'} - f(V) \). Hence,
\( f(U \cup \{\omega'\}) - f(U) \ge f(V \cup \{\omega'\}) - f(V) \).
In both cases, the diminishing returns condition holds, confirming that \( f \) is submodular. \hfill \Halmos \end{proof} 

\begin{proof}{Proof of Lemma \ref{lem:sfree}.}
Let $(y,z)\in\mathrm{epi}(f)$. By definition of $\mathrm{epi}(f)$, we have
$z\in\{0,1\}^N$ and $y\ge f(z)$.
Take any $\pi\in\mathrm{EPM}_f$. By \eqref{eq:epm},
$\pi^\top z\le f(z)$. Hence
$y\ge f(z)\ge \pi^\top z$. So, $\mathcal{H}_{\pi}\cap\mathrm{epi}(f) \neq \emptyset$ unless $y=\pi^\top z$. Therefore, 
$\intset\big(\mathcal{H}_{\pi}\big)\cap\mathrm{epi}(f)=\emptyset$. Since
$\mathcal{H}_{\pi}$ is a closed half-space, it is convex.
\hfill \Halmos \end{proof}

\begin{proof}{Proof of Lemma \ref{lemma:ray_exit}.}
For each non-basic ray \( j \in J \), define the intersection function
\begin{equation} \label{eq:intersection_function}
    \zeta^j(\eta)
    = \bigl(\bar{y} + \eta r_y^j\bigr)
      - \bar{\pi} ^ \top \bigl(\bar{z} + \eta r_z^j\bigr), \quad \eta \ge 0.
\end{equation}
The ray exit distance is \( \eta_j^* = \sup\{\eta > 0 : \zeta^j(\eta) \le 0\} \). 
Given that $\zeta^j(\eta)$ is linear, setting \( \zeta^j(\eta) = 0 \) yields the unique crossing point. If 
$r_y^j - \bar{\pi}^\top  r_z^j > 0$, then \( \eta_j^* > 0 \), as defined in \eqref{eq:eta_sa}. Otherwise, the ray cannot exit
\( \mathcal{H}_{\bar{\pi}} \); hence, \( \eta_j^* = +\infty \).
\hfill \Halmos \end{proof}

\begin{proof}{Proof of Theorem \ref{thm:ic_sa}.}
By Lemma~\ref{lem:sfree}, \(\mathcal{H}_{\bar \pi}\) is an S-free set for \eqref{eq:genprob_chance}.
Lemma~\ref{lemma:ray_exit} provides the ray-exit distances \(\eta_j^*\) from
\((\bar{y},\bar{z})\in\intset(\mathcal{H}_{\bar \pi})\) along each non-basic ray.
We now map the auxiliary variables \( (y, z) \) and the corresponding ray directions
\( (r_y^j, r_z^j) \), $j \in J$, back to the original decision variables \( (x, \beta) \) using the
transformations \( y = A_i x \) and \( z = \mathbf{1} - \beta \). This yields
\begin{align*}
    \bar{y} = A_i \bar{x}, \qquad r_y^j &= A_i r_x^j, \\
    \bar{z} = \mathbf{1} - \bar{\beta}, \qquad r_z^j &= -r_\beta^j. 
\end{align*}
whose substitution in \eqref{eq:eta_sa} yields
\begin{equation*} 
    \eta_j^*
    = \frac{\bar{\pi}^\top(\mathbf{1} - \bar{\beta}) - A_i \bar{x}}
           {A_i r_x^j + \bar{\pi}^\top r_\beta^j}.
\end{equation*}
Now, a direct application of Theorem \ref{thm:ic_prelim} yields \eqref{eq:ic_sa}. 
\hfill \Halmos \end{proof}

\begin{proof}{Proof of Proposition \ref{prop:modularity}.}
Since \( f_K \) is a nonnegative linear function of \( z \), its marginal gain upon adding any
element \( k \notin U \) for any set \( U \subseteq \Omega \) is identically \( 1 \) if
\( k \in K \) and \( 0 \) otherwise. Therefore, the diminishing returns
condition holds with equality for all \( U \subseteq V \subseteq \Omega \), confirming
modularity.
\hfill \Halmos \end{proof}

\begin{proof}{Proof of Lemma \ref{lem:cover_sfree}.}
Let \((y,\beta) \in F\). By definition,
\(\beta \in \{0,1\}^{N}\) and \(\sum_{\omega \in \Omega} p_{\omega} \beta_{\omega} \le \epsilon\). 
If \(\beta_\omega = 1\) for all \(\omega \in K\), then
\(\sum_{\omega \in \Omega} p_\omega \beta_\omega \ge \sum_{\omega \in K} p_\omega > \epsilon\),
which contradicts feasibility. Hence, there exists at least one
\(\omega \in K\) with \(\beta_\omega = 0\), and therefore
\(\sum_{\omega \in K} \beta_\omega \le |K|-1\). Thus,
\((y,\beta) \notin \mathcal{H}_K\) unless $\sum_{\omega \in K} \beta_\omega \le |K|-1$. Consequently, \(\intset\big(\mathcal{H}_K\big) \cap F = \emptyset\). 
Since \(\mathcal{H}_{K}\) is a closed half-space, it is convex. 
\hfill \Halmos \end{proof}

\begin{proof}{Proof of Lemma \ref{lem:eta_ma}.}
For each non-basic ray \( j \in J \), define the intersection function along the ray trajectory as
\begin{equation} \label{eq:zeta_modular}
    \zeta_K^j(\eta)
    =  (|K| - 1) - \sum_{\omega \in K} \bigl(\bar{\beta}_\omega  + \eta\, r_{\beta,\omega}^j\bigr) 
    = \Delta_K - \eta \sum_{\omega \in K} r_{\beta,\omega}^j,
\end{equation}
where \(\Delta_K := (|K| - 1) - \sum_{\omega \in K} \bar{\beta}_\omega   <0\). 
The ray exit distance is
\( \eta_j^* = \sup\{\eta > 0 : \zeta_K^j(\eta) \le 0\} \).
Given that $\zeta_K^j(\eta)$ is linear, setting \( \zeta_K^j(\eta) = 0 \) yields the unique crossing point. If 
\( \sum_{\omega \in K} r_{\beta,\omega}^j < 0 \), then \( \eta_j^* > 0 \), as defined in \eqref{eq:eta_ma}. Otherwise, the ray trajectory does not decrease the violation depth and cannot exit
\( \mathcal{H}_K \); hence, \( \eta_j^* = +\infty \).
\hfill \Halmos \end{proof}

\begin{proof}{Proof of Theorem \ref{thm:ic_ma}.}
By Lemma~\ref{lem:cover_sfree}, \(\mathcal{H}_K\) is a closed convex S-free set for the integer-feasible set.
Lemma~\ref{lem:eta_ma} provides the ray-exit distances \(\eta_j^*\) from
\((\bar{x},\bar{\beta})\in\intset(\mathcal{H}_K)\) along each non-basic ray.
Now, a direct application of Theorem \ref{thm:ic_prelim} yields \eqref{eq:ic_ma_main}. \hfill \Halmos \end{proof}

\subsection{Proofs of Section \ref{sec:recourse}}

\begin{proof}{Proof of Proposition \ref{prop:wellposed}.}
Feasibility follows from \(P_\omega \cap \bar{X} \neq \emptyset\).
Boundedness follows because $\alpha \in C^*$ implies $\alpha^\top r \geq 0$
for all $r \in C$ and that the recession cone of $ P_\omega \cap \bar{X}$ is contained in $C$, where $C$ is the common recession cone of
$P_\omega$. 
For validity, take any \(x \in P_\omega \cap \bar{X}\). By definition of
\eqref{eq:h_omega}, we have \( \alpha^\top x \ge u^\omega(\alpha)\).
\hfill \Halmos \end{proof}

\section{Missing Separation Routines}

\subsection{Submodular Intersection Cut}
\label{sec:EC_ic_sa}

Algorithms \ref{algo:ic_sa_greedy} and \ref{algo:ic_sa_sep} present the missing separation routines for the submodular intersection cut, developed in Section \ref{sec:ic_sa}, for the non-recourse CCPs, defined via \eqref{eq:Pk_nonrecourse}.

\begin{algorithm}[!tb]
\caption{Lov\'{a}sz $\textrm{Greedy}(\bar z, h)$.}
\label{algo:ic_sa_greedy}
\begin{algorithmic}[1]

\State
Let $\vartheta$ be a permutation of $\Omega$ describing a nonincreasing order of $\bar z_\omega$, $\omega \in \Omega$, i.e., $\bar z_{\vartheta(1)} \ge \bar z_{\vartheta(2)} \ge  \ldots \ge \bar z_{\vartheta(N)}$. 

\For{$k \in  [N]$}
    \State{$U \leftarrow \{\vartheta(1), \ldots, \vartheta(k-1)\}$}
    
    \State {
       $\pi_{\vartheta(k)} := f(U\cup \{\vartheta(k)\}) -f (U)$, where $f(U)=\max_{\omega \in U} h^{\omega}$ and $f(\emptyset)=0$.}

\EndFor

\State \Return{$\bar{\nu} \leftarrow \pi^\top \bar{z}$ and $\bar \pi \leftarrow \pi$.}

\end{algorithmic}
\end{algorithm}

\begin{algorithm}[!tb]
\caption{Cut separation routine $\textrm{SepCut}_{SA}(\bar x, \bar \beta, h, \bar y, \bar z, \bar \pi, \bar \nu, J, r)$.}
\label{algo:ic_sa_sep}
\begin{algorithmic}[1]

\State $\textit{cut} \leftarrow 0$, $\bar z^0 \leftarrow \bar z$, $\bar y^0 \leftarrow \bar y$.

\For{$j \in J$}

    \State $\lambda_j \leftarrow 0$, $\eta_j^* \leftarrow +\infty$, $\zeta_j  \leftarrow \bar{y} - \bar{\nu}$, $\pi^j \leftarrow \bar \pi$. 
     \While{$\zeta_j < 0$}
        
        \State Compute directional derivative 
               $D \leftarrow r_y^j - (\pi^j)^\top r_z^j$.
               \Comment{Local slope of $\zeta^j(\eta)$ on current facet}

        \If{$D \le 0$}
            \State $\lambda_j \leftarrow +\infty$; \textbf{break}.
            \Comment{Ray moves deeper into $\mathcal{H}_{\bar{\pi}}$; no intersection}
        \EndIf

        \State $\lambda_j \leftarrow \lambda_j - \zeta_j / D$.
               \Comment{Newton step: $\zeta_j + D \cdot \Delta\lambda_j = 0 \Rightarrow \Delta\lambda_j = -\zeta_j/D$}

        \State $\bar {z}^j \leftarrow \bar{z}^0 + \lambda_j r_z^j$, $\bar {y}^j \leftarrow \bar{y}^0 + \lambda_j r_y^j$.

        \State $(\bar \nu, \pi^j) \leftarrow$ Lov\'{a}sz $\textrm{Greedy}(\bar z^j, h)$ 
        \Comment{new $\bar{\nu}$ and $\pi^j$ at the updated point $\bar{z}$}

        \State $\zeta_j \leftarrow \bar{y}^j 
               - \bar{\nu}$.
               
    \EndWhile

    \State $\eta_j^* \leftarrow \lambda_j$.

    \If{$\eta_j^* = +\infty$} 
        \State $\psi_j \leftarrow 0$.
    \Else 
        \State $\psi_j \leftarrow 1 / \eta_j^*$.
    \EndIf

    \State Retrieve tableau mapping $\varphi^j(x, \beta)$ for non-basic variable $s_j$.
    \State $\textit{cut} \leftarrow \textit{cut} + \psi_j \cdot \varphi^j(x, \beta)$.

\EndFor

\State \Return $\textit{cut} \ge 1$. 

\end{algorithmic}
\end{algorithm}

\subsection{Modular Intersection Cut}
\label{sec:EC_ic_ma}

Algorithms \ref{algo:ic_ma_cover} and \ref{algo:ic_ma_sep} present the missing separation routines for the modular intersection cut, developed in Section \ref{sec:ic_ma}, for the non-recourse CCPs, defined via \eqref{eq:Pk_nonrecourse}. 

\begin{algorithm}[!tb]
\caption{Minimal $\textrm{Cover}(\bar \beta)$.}
\label{algo:ic_ma_cover}
\begin{algorithmic}[1]

\State  $K \leftarrow \emptyset$,  $\textit{mass} \leftarrow 0$, $\textit{depth} \leftarrow 0$.

\State
Let $\vartheta$ be a permutation of $\Omega$ describing a nonincreasing order of $\bar \beta_\omega$, $\omega \in \Omega$, i.e., $\bar \beta_{\vartheta(1)} \ge \bar \beta_{\vartheta(2)} \ge  \ldots \ge \bar \beta_{\vartheta(N)}$. 

\For{$k \in [N]$}
    \State{$K \leftarrow K \cup \{\vartheta(k)\}$}; $\textit{mass} \leftarrow \textit{mass} + p_{\vartheta(k)}$;\;
           $\textit{depth} \leftarrow \textit{depth} + \bar{\beta}_{\vartheta(k)}$.

    \If{$\textit{mass} > \epsilon$}
        \State $\Delta_K \leftarrow (|K| - 1) - \textit{depth}$
        \State \Return $K$ and $\Delta_K$; \textbf{break}.
    \EndIf
    
\EndFor

\end{algorithmic}
\end{algorithm}

\begin{algorithm}[!tb]
\caption{Cut separation routine $\textrm{SepCut}_{MA}(\bar x, \bar \beta, K, \Delta_K, J, r)$.}
\label{algo:ic_ma_sep}
\begin{algorithmic}[1]
\State $\textit{cut} \leftarrow 0$.
\For{$j \in J$}
    \State $\lambda_j \leftarrow 0$, $\eta_j^* \leftarrow +\infty$.  
    \State Compute directional derivative 
               $D \leftarrow - \sum_{\omega \in K}  r_{\beta,\omega}^j$.
        \If{$D \le 0$}
            \State $\lambda_j \leftarrow +\infty$. 
        \Else
        \State $\lambda_j \leftarrow -\frac{\Delta_K}{D}$;
        \EndIf
    \State $\eta_j^* \leftarrow \lambda_j$. 
    \If{$\eta_j^* = +\infty$} 
        \State $\psi_j \leftarrow 0$.
    \Else 
        \State $\psi_j \leftarrow 1 / \eta_j^*$.
    \EndIf
    \State Retrieve tableau mapping $\varphi^j(x, \beta)$ for non-basic variable $s_j$.
    \State $\textit{cut} \leftarrow \textit{cut} + \psi_j \cdot \varphi^j(x, \beta)$.
\EndFor
\State \Return $\textit{cut} \ge 1$. 
\end{algorithmic}
\end{algorithm}

\section{Test Instance Data Generation}
\label{sec:EC_data_generation}

To comprehensively evaluate the robustness of the proposed cutting-plane policies under
both non-recourse and recourse paradigms, we implement a scenario-based data generation
scheme following the framework established by \cite{luedtke2014branch}. This scheme
explicitly captures the correlation between resource efficiency and cost, as well as
compatibility restrictions between manufacturers and retailers.

For a given problem size $(|\mathcal{I}|, |\mathcal{J}|)$, a base instance is constructed
as follows. The unit cost $c_i$ for each manufacturer $i \in \mathcal{I}$ is drawn
independently from a normal distribution with mean $1$ and standard deviation $0.2$,
i.e., $c_i \sim \mathcal{N}(1,\, 0.2^2)$. These unit costs simultaneously serve as base
efficiency rates, enforcing the realistic correlation that higher-cost manufacturers
exhibit greater operational efficiency. A base service rate $\mu'_j$ for each retailer
$j \in \mathcal{J}$ is independently drawn from $\mathcal{N}(1,\, 0.2^2)$. The effective
delivery rate of manufacturer $i$ for retailer $j$ is then computed as 
\begin{equation} \label{eq:dij}
    d_{ij} = \min\{1,\,\max\{0,\, c_i\mu'_j\}\}.
\end{equation}

To reflect the realistic sparsity of supply chain compatibility, a fraction of the
$d_{ij}$ values is set to zero. Each retailer $j$ is independently categorized as
``difficult'' (with probability $0.5$) or ``easy'' (with probability $0.5$). For
difficult retailers, each $d_{ij}$ is independently zeroed with probability $0.6$;
for easy retailers, the zeroing probability is $0.3$. To prevent pathological instances
in which a retailer has no compatible manufacturers, zeroing is halted for retailer $j$
once more than $\lfloor 0.7 |\mathcal{I}| \rfloor$ of its service rates have been
set to zero.

For each base instance, $N$ independent demand scenarios are generated to form the finite
support $\Omega$. The random demand vector $r \in \mathbb{R}^{|\mathcal{J}|}$
is drawn from a multivariate normal distribution $r \sim \mathcal{N}(\bar{\lambda}, V)$
to capture correlated demand surges across retailers. The mean demand for
each retailer $j$, $\bar \lambda_j$, is independently drawn from $\mathcal{N}(110,\, 25^2)$. The covariance
matrix is constructed as 
\begin{equation} \label{eq:cov}
    V = \frac{1.25}{|\mathcal{J}|}\, \tilde{v} \tilde{v}^\top,
\end{equation}
where the entries of $\tilde{v} \in \mathbb{R}^{|\mathcal{J}|}$ are drawn independently
from $\mathcal{U}[-6.25,\, 25]$.

\end{document}